\documentclass[a4paper,12pt]{amsart}
\usepackage{amsfonts,amsthm,amsmath,amssymb,graphicx}
\usepackage[utf8]{inputenc}
\usepackage[english]{babel}
 \usepackage[utf8]{inputenc}
\usepackage{float}
\usepackage{hyperref}
\usepackage[T1]{fontenc}
\usepackage{tikz}
\usetikzlibrary{arrows}
    \tikzset{point/.style={circle,inner sep=0pt,minimum size=3pt,fill=red}}
\usetikzlibrary{shapes.multipart}
\usetikzlibrary{matrix}
\usetikzlibrary{intersections}
\usetikzlibrary{decorations.markings}
\usetikzlibrary{positioning}
\usetikzlibrary{calc}
\usetikzlibrary{decorations.pathreplacing,shapes.misc}
\usepackage{geometry}
\usetikzlibrary{shapes.geometric}
\usetikzlibrary{decorations.pathmorphing}

\theoremstyle{plain}
\newtheorem{lem}{Lemma}[section]
\newtheorem{prop}[lem]{Proposition}
\newtheorem{thm}[lem]{Theorem}

\newtheorem{rem}{Remark}

\theoremstyle{definition}

\newtheorem{ex}[lem]{Example}

\theoremstyle{remark}

\numberwithin{equation}{section}

\DeclareMathOperator{\dom}{dom}
\DeclareMathOperator{\ext}{ext}
\DeclareMathOperator{\Eq}{Eq}

\newcommand{\R}{\mathbb R}
\newcommand{\Z}{\mathbb Z}
\newcommand{\N}{\mathbb N}

\renewcommand{\epsilon}{\varepsilon}

\title{ZERO TEMPERATURE LIMITS OF Equilibrium STATES FOR SUBADDITIVE POTENTIALS and approximation of maximal Lyapunov exponent}
\author{reza mohammadpour }

\address{Department of Dynamical systems, Institute of Mathematics of the Polish Academy of Sciences, ul.~\'Sniadeckich 8, 00-656 Warsaw, Poland}
\date{\today}

\email{rmohammadpour@impan.pl}
\begin{document}
\title[ZERO TEMPERATURE LIMITS OF Equilibrium STATES]{ZERO TEMPERATURE LIMITS OF Equilibrium STATES FOR SUBADDITIVE POTENTIALS and approximation of the maximal Lyapunov exponent}

\maketitle

\begin{abstract}
In this paper we study ergodic optimization problems for subadditive sequences of functions on a topological dynamical system.  We prove that for $t\rightarrow \infty$ any accumulation point of a family of equilibrium states is a maximizing measure. We show that the Lyapunov exponent and entropy of equilibrium states converge in the limit $t\rightarrow \infty$ to the maximum Lyapunov exponent and entropy of maximizing measures.\\
In the particular case of matrix cocycles we prove that the maximal Lyapunov exponent can be approximated by Lyapunov exponents of periodic trajectories under certain assumptions.
\end{abstract}

\section{Introduction and statement of the results}
Throughout this paper $X$ is a compact metric space that is endowed with the metric $d$. We call $(X, T)$ a \textit{topological dynamical system} (TDS), if $T:X \rightarrow X$ is a continuous map on the compact metric space $X$. We say that $\Phi:=\{\log \phi_{n}\}_{n=1}^{\infty}$  is a \textit{subadditive potential} if each $\phi_{n}$ is a continuous non-negative-valued function on $X$ such that
\[ 0\leq \phi_{n+m}(x) \leq \phi_{n}(x) \phi_{m}(T^{n}(x)) \hspace{0,2cm} \forall x\in X, m,n \in \N.\]

Furthermore,  $\Phi=\{\log\phi_{n}\}_{n=1}^{\infty}$ is said to be an $\textit{almost additive potential}$ if there exists a constant $C > 0$ such that for any $m,n \in \N$, $x\in X$, we have
\[ C^{-1}\phi_{n}(x)\phi_{m}(T^{n})(x) \leq \phi_{n+m}(x)\leq C \phi_{n}(x) \phi_{m}(T^{n}(x)).\]

For any $T-$invariant measure $\mu$ such that $\log^{+} \phi_{1} \in L^{1}(\mu)$, the \textit{pointwise Lyapunov exponent}
\[ \chi(x, \Phi):= \lim_{n\rightarrow \infty}\frac{1}{n} \log \phi_{n}(x)\in [-\infty, \infty),\]
exists for a.e. point $x$.

By Kingman's subadditive ergodic theorem \cite[Theorem 3.3]{V}, the \textit{Lyapunov exponent of measure $\mu$} \[\chi(\mu, \Phi):=\lim_{n\rightarrow \infty}\frac{1}{n} \int \log \phi_{n}(x) d\mu(x)\] exists. If $\mu$ is ergodic then $\chi(x, \Phi)=\chi(\mu, \Phi)$ for $\mu-$a.e. point $x$. We remark that Furstenberg and Kesten \cite{FK1} first proved, under a suitable integrability assumption, the existence of maximal and minimal Lyapunov exponents for matrix cocycles. Indeed, their result is a straightforward consequence of the Kingman's subadditive ergodic theorem, where it was proved later.

In this paper, we are interested in the \textit{maximal Lyapunov exponent}, defined as 
\[ \beta(\Phi):=\lim_{n\rightarrow \infty}\frac{1}{n}\log \sup_{x\in X} \phi_{n}(x).\]

 We denote by $\mathcal{M}(X,T)$ the space of all $T-$invariant Borel probabilty measures on X. Morris \cite{M2} showed that one can define the maximal Lyapunov exponent as the supremum of the Lyapunov exponents of measures over invariant measures. That means,
\begin{equation}\label{invariant1}
\beta(\Phi)=\sup_{\mu \in \mathcal{M}(X,T)} \chi(\mu, \Phi).
\end{equation}
Feng and Huang \cite{FH} gave a different proof of it.

Let us define the set of \textit{maximizing measures} of $\Phi$ to be the set of measures on $X$ given by 
\[ \mathcal{M}_{\max}(\Phi):=\{\mu \in \mathcal{M}(X, T), \hspace{0,2cm} \beta(\Phi)=\chi(\mu, \Phi) \}. \]

In this paper, we study the behavior of the equilibrium measures $(\mu_{t})$ for the subadditive potentials $t\Phi$ when $t\rightarrow \infty$. In the thermodynamic interpretation of the parameter $t$, it is the \textit{inverse temperature}. The limits $t\rightarrow \infty$ are called \textit{zero temperature limits}, and the accumulation points of the measures $(\mu_{t})$ as $t\rightarrow \infty$ are called \textit{ground states}.

The topic of \textit{ergodic optimization} revolves around realizing invariant measures which maximize the Lyapunov exponents. Zero temperature limits laws are also related to ergodic optimization, because for $t\rightarrow \infty$ any accumulation point of the equilibrium measure $(\mu_{t})$ will be a maximizing measure $\Phi$ (maximizing $\chi(\mu, \Phi))$. We refer the reader to \cite{BG} and \cite{J}.

The behavior of the equilibrium measure $(\mu_{t})$ as $t\rightarrow \infty$ has also been analyzed. In particular, the continuities of zero temperature limits $(\mu_{t})_{t\rightarrow \infty}$ in the sense,
\begin{equation}\label{eq2}
 \chi(\mu, \Phi)=\lim_{t\rightarrow \infty}\chi(\mu_{t}, \Phi),
\end{equation} 
and
\begin{equation}\label{eq3}
h_{\mu}(T)=\lim_{t\rightarrow \infty}h_{\mu_{t}}(T),
\end{equation}
have been investigated by many authors  \cite{DUZ}, \cite{IY}, \cite{J}, \cite{JMU}, \cite{M1}, \cite{WZ}, \cite{Z}.

In the non-compact space setting, (\ref{eq2}) and (\ref{eq3}) were proved by Jenkinson, Mauldin and Urba\'nski \cite{JMU}, and  Morris \cite{M1} on the additive potential $\psi:X \rightarrow \R$. In fact, the proof of Theorem \ref{main1} will be based on ideas from those works. Moreover, this kind of result is known for almost subadditive potentials by Zhao \cite{Z} under the specification property, upper semi-continuity of entropy and finite topological entropy assumptions.

The goal of this paper is to extend the above results for subadditive potentials.

Note that even though we know the existence of an accumulation point for the sequence $(\mu_{t})$ (see Proposition \ref{eq11}), this does not imply that the $\lim_{t\rightarrow \infty}\mu_{t}$ exists. In fact, Chazottes and Hochman \cite{CH} constructed an example on compact sub-shifts of finite type and H\"older potentials, where there is no convergence. For more information about zero temperature see \cite{J}.

Our main results are Theorems $1.1$ and $1.3$ formulated as follows:
\begin{thm}\label{main1}
Let $(X,T)$ be a TDS such that the entropy map $\mu \mapsto h_{\mu}(T)$ is upper semi-continuous and topological entropy $h_{top}(T)<\infty.$ Suppose that $\Phi=\{\log \phi_{n}\}_{n=1}^{\infty}$ is a subadditive potential on the compact metric space $X$ which satisfies $\beta(\Phi)>-\infty.$ Then any weak$^{\ast}$ accumulation point $\mu$ of a family of equilibrium measures $(\mu_{t})$ for potentials $t\Phi$, where $t>0$, is a Lyapunov maximizing measure for $\Phi$. Moreover,
\begin{itemize}
\item[(i)] $\chi(\mu, \Phi)=\lim_{t\rightarrow \infty} \chi(\mu_{t}, \Phi),$
\item[(ii)]$h_{\mu}(T)=\lim_{t \rightarrow \infty} h_{\mu_{t}}(T)=\max\{h_{\nu}(T), \nu \in \mathcal{M}_{\max}(\Phi)\}.$
\end{itemize}
Furthermore, $\beta(\Phi)$ can be approximated by 
 Lyapunov exponents of equilibrium measures of a subadditive potential $t\Phi$.
\end{thm}
 Let $\mathcal{A}:X \rightarrow GL(d, \R)$ be a measurable function. We can define a \textit{linear cocycle} $F:X \times \R^{d} \rightarrow X\times \R^{d}$  as
\[ F(x, v)=(T(x), \mathcal{A}(x)v).\]

We say that $F$ is generated by $T$ and $\mathcal{A}$, we will also denote by $(T, \mathcal{A}).$  Observe that $F^{n}(x, v)=(T^{n}(x), \mathcal{A}_{n}(x)v)$ for each $n \geq 1$, where
\[ \mathcal{A}_{n}(x)=\mathcal{A}(T^{n-1}(x))\mathcal{A}(T^{n-2}(x))\ldots\mathcal{A}(x).\]

If $T$ is invertible then so is $F$. Moreover, $F^{-n}(x)=(T^{-n}(x), \mathcal{A}_{-n}(x)v)$ for each $n\geq1$, where
\[\mathcal{A}_{-n}(x)=\mathcal{A}(T^{-n}(x))^{-1}\mathcal{A}(T^{-n+1}(x))^{-1}\ldots\mathcal{A}(T^{-1}(x))^{-1}.\]

A special class of linear cocycles is a class of \textit{locally constant cocycles} which is defined as follows.
\begin{ex}\label{ex1}
Let $X=\{1,\ldots,k\}^{\Z}$ be a symbolic space. Let $T:X \rightarrow X$  be a shift map, i.e. $T(x_{l})_{l}=(x_{l+1})_{l}$. Given a finite set of matrices $\mathcal{A}=\{A_{1},\ldots,A_{k}\}\subset GL(d, \R)$, we define the function $\mathcal{A}:X \rightarrow GL(d, \R)$ by $\mathcal{A}(x_{l})_{l}=A_{x_{0}}$. 
\end{ex}
We say that a homeomorphism $T$ satisfies the \textit{Anosov closing property} if there exists $C, \epsilon, \delta>0$ such that for any $x\in X$ and $n\in \N$ with $d(x, T^{n}(x))< \epsilon$ there exists a point $p \in X$ with $T^{n}(p)=p$ such that the orbits $\mathcal{O}^{+}(T(x)):=\{T^{k}(x), k\in \N \}$, and $\mathcal{O}^{+}(T(p)):=\{T^{k}(p), k\in \N \}$ are exponentially close, i.e.
\[ d(T^{i}(x), T^{i}(p)) \leq C e^{-\delta \min\{i, n-i\}} d(T^{n}(x), x)\]
for every $i=0,\ldots,n.$

Note that shifts of finite type, Axiom A diffeomorphisms, and hyperbolic homeomorphisms are particular systems satisfying the Anosov closing property. See for more information \cite{KH}.

Kalinin and Sadovskaya \cite{KS} proved that if a homeomorphism $T$ satisfies the Anosov closing property, and $\mathcal{A}: X \rightarrow GL(d, \R)$ is a H\"older continuous Banach cocycle, then the maximal Lyapunov exponent can be approximated by Lyapunov exponents of measures supported on periodic orbits. For a locally constant cocycle $(T, \mathcal{A})$, where $\mathcal{A}: X \rightarrow GL(2, \R)$, we show that the maximal Lyapunov exponent can be approximated by Lyapunov exponents of measures supported on periodic orbits. Theorem \ref{main3} is implied by the Kalinin and Sadovskaya's result. However, the methods of proof are different.

We write $\phi_{n}:=\|\mathcal{A}_{n}\|$, where $\| \|$ is the operator norm.

\begin{thm}\label{main3}
Let $(T, \mathcal{A})$ be a locally constant cocycle satisfying the Anosov closing property. Then the maximal Lyapunov exponent $\beta(\Phi)$ can be approximated by Lyapunov exponents of measures supported on periodic orbits.
\end{thm}

In general, Kalinin \cite{Ka} showed that for a H\"older continuous map $\mathcal{A}:X \rightarrow GL(d, \R)$, Lyapunov exponents can be approximated by Lyapunov exponents of measures supported on periodic orbits under an assumption slightly stronger than the Anosov closing property.

This paper is organized as follows. In Section 2, we recall some preliminary material regarding convex functions as well as some results in thermodynamic formalism for subadditive setting.
In Section 3, we prove Theorem \ref{main1}. In Section 4, we state a theorem about the continuity of Lyapunov exponents for locally constant cocycles, and we prove Theorem \ref{main3}.

\textbf{\textit{Acknowledgements.}}
The author thanks M. Rams for his careful reading of an earlier version of this paper and many helpful suggestions.  The author was partially supported by the National Science Center grant 2014/13/B/ST1/01033 (Poland).

\section{Preliminaries}
\subsection{Convex functions}
We first give some notation and basic facts in convex analysis.  For details, one is referred to \cite{HL}. 

Let $x, y \in\R^{n}$, the line segment connecting $x$ and $y$ is the set $[x, y]$ formally given by
\[ [x, y] =\{\beta x+ (1-\beta)y \hspace{0,2cm} \beta \in[0,1]\}.\]

We say that a set $X\subset \R^{n}$ is convex when for any two points $x, y\in X$, the line segment $[x, y]$ also belongs to the set $X$, i.e., $\beta x+ (1-\beta)y \in X$ for any $x, y\in X$ and $\beta\in(0,1).$ Let $C$ be  a  convex  subset  of $\R^n$.   A  point $x\in C$ is  called  an \textit{extreme point} of $C$ if whenever $x=\beta y+ (1-\beta)z$ for  some $y,z\in C$ and  $0< \beta<1$,  then $x=y=z$. We denote by $\ext(C)$ the  set  of extreme points of $C$. 

A function $f:\R^{n} \rightarrow \R$ is a convex function if its domain $\dom(f)$ is a convex set and for all $x_, y\in \dom(f)$ and $\beta \in (0,1)$, the following relation holds 
\[ f(\beta x+ (1-\beta) y))\leq \beta f(x) + (1-\beta)f(y).\]

In other words, a function $f:\R^{n} \rightarrow \R$ is convex when for every segment $[x_{1}, x_{2}]$, as the vector $x_{\beta}=\beta x_{1}+(1-\beta)x_{2}$ varies within the line segment $[x_{1}, x_{2}]$, the points $(x_{\beta}, f(x_{\beta}))$ on the graph $\{(x, f(x)) | x\in \R^{n}\}$ lie below the segment connecting $(x_{1}, f(x_{1}))$ and $(x_{2}, f(x_{2}))$, as illustrated in Figure \ref{fig:M1}.
\begin{figure}[ht]
\centering
  \begin{tikzpicture}

\draw (2,1) -- (8,1)    ;

\draw (2.5,0) -- (2.5,8)    ;
\coordinate (G) at (1.5,3.0);
\coordinate (R) at (2.0,5);
\coordinate (B) at (7.4,5);

\coordinate (D) at (8.2,5);
\draw    (R) to[out=-90,in=-95] (B) ;
\draw  (G) -- (D)node[pos=0.14,above]{\((x_{1}, f(x_{1}))\)}node[pos=0.87,above]{\((x_{2}, f(x_{2}))\)};

\coordinate [label=above:${(x_{\beta} ,f(x_{\beta})) }$] (C) at (4.9,3.45);
\node[point] at (C) {};

\end{tikzpicture}
    \caption{Convex line} \label{fig:M1}
    \label{tikz:f}
\end{figure}
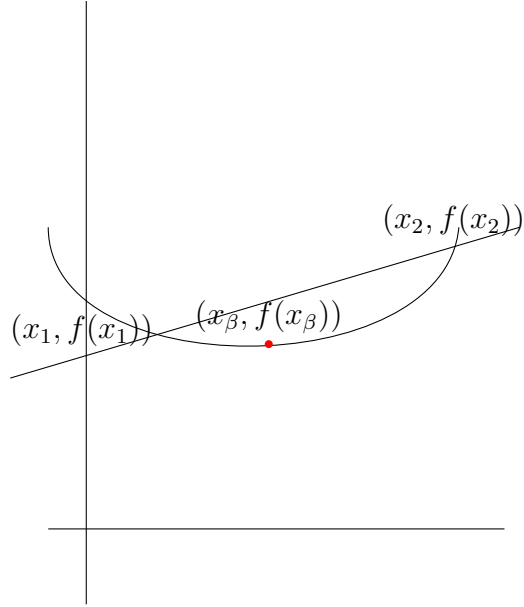

Let $U$ be an open convex subset of $\R^{n}$ and $f$ be a real continuous convex function on $U$. We say a vector $a\in \R^{n}$ is a \textit{subgradient} of $f$ at $x$ if for all $z \in U$, 
\[f(z)\geq f(x) +a^{T}(z-x),\]
where the right hand side is the scalar product.

For each $x\in \R^{n}$ set the \textit{subdifferential} of $f$ at a point $x$ to be
\[\partial f(x) :=\{a: a \hspace{0,1cm}\textrm{is a subgradient for} \hspace{0,1cm}f\hspace{0,1cm} \textrm{at}\hspace{0,1cm} x \}.\]

 For $x\in U$, the subdifferential $\partial f(x)$ is
 a nonempty convex compact set.  Define $\partial^{e} f(x) := \ext\{\partial f(x)\}$. In the case $n=1$, $\partial^{e} f(x)=\{f^{'}(x_{-}), f^{'}(x_{+})\}$, where $f^{'}(x_{-})$ (resp. $f^{'}(x_{+}))$ denotes the left (resp. right) derivative. We say that $f$ is \textit{differentiable} at $x$ when  $\partial^{e} f(x)=\{a\}$. 

We define 
\begin{equation}\label{dense}
\partial f(U)=\cup_{x \in U} \partial f(x) \hspace{0,2cm}\textrm{and}\hspace{0.2cm} \partial^{e} f(U)=\cup_{x \in U} \partial^{e} f(x).
\end{equation} 

In the case $n = 1$,  
Lebesgue's theorem on the differentiability of monotone functions says $\partial^{e} f$ is  differentiable almost everywhere. The case $n= 2$ was proven by H. Busemann and W. Feller \cite{BF}. The general case was settled by A. D. Alexandrov \cite{A}. The following result is well known (cf. \cite[Theorem 7.9]{S}).
\begin{thm}\label{incre1}Let $f$ be a continuous function defined on an open interval that has a derivative at each point of $\R$ except on a countable set, and $f^{'} \leq 0$ Lebesgue almost everywhere, then $f$ is a non-increasing function.
\end{thm}
\subsection{Thermodynamic formalism for a subadditive potential}
 We require some elements from the subadditive thermodynamic formalism. The additive theory of thermodynamic
formalism extends to the subadditive theory with suitable generalizations. Let $(X,T)$ be a TDS and let $\Phi=\{\log \phi_{n}\}_{n=1}^{\infty}$ be a subadditive potential on $(X,T)$. 
 
 We introduce the \textit{topological pressure} of $\Phi$ as follows. The space $X$ is endowed with the metric $d$. For any $n\in \N$, one can define a new metric $d_{n}$ on $X$ by
\[ d_{n}(x, y)=\max\left\{ d(T^{k}(x), T^{k}(y)) : k=0,\ldots,n-1  \right\}.\]

For any $\epsilon>0$ a set $E \subset X$ is said to be a $(n,\epsilon)$-\textit{separated  subset}  of $X$ if $d_{n}(x,y)> \epsilon$ for any two different points $x,y \in E$.  We define for $\Phi$
\[ P_{n}(T, \Phi, \epsilon)=\sup \left\{\sum_{x\in E} \phi_{n}(x) : E \hspace{0,1cm}\textrm{is} \hspace{0,1cm}(n, \epsilon) \textrm{-separated subset of X} \right\}.\]
Since $P_{n}(T, \Phi, \epsilon)$ is a decreasing function of $\epsilon$, we define 
 \[P(T, \Phi, \epsilon)=\limsup_{n \rightarrow \infty} \frac{1}{n} \log P_{n}(T, \Phi, \epsilon),\] and
\[ P(T, \Phi)=\lim_{\epsilon \rightarrow 0}P(T, \Phi, \epsilon).\]

We call $P(T,\Phi)$ the topological pressure of $\Phi$. We define $h_{top}(T):=P(T, 1).$

Bowen \cite{B2} showed that for any H\"older continuous $\psi:X\rightarrow \R$ on a mixing hyperbolic system, there exists a unique equilibrium measure (which is also a Gibbs state) for $\psi$.

Feng and K\"aenm\"aki \cite{FK} extended the Bowen's result for the subadditive potentials $t\Phi$ on a locally constant cocycle under the assumption that the matrices in $\mathcal{A}$ do not preserve a common proper subspace of $\R^{d}$ (i.e. $(T, \mathcal{A})$ is irreducible).

Recently, Park \cite{P} showed the continuity of the topological pressure, and the uniqueness of the equilibrium measure for general cocycles under generic assumptions. In \cite{M} the continuity of the topological pressure was proven under some assumption which is weaker than Park's assumptions.

Let $(X, \tau, \mu)$ be a Borel probability space, and $T:X\rightarrow X$ be a measure preserving transformation.

 A \textit{partition} of $(X, \tau, \mu)$ is a subfamily of $\tau$ consisting of mutually disjoint elements whose union is $X$. We denote by $\alpha$ and $\beta$ the countable partition of $X$.
 
 Let $\alpha=\{A_i, i\geq1\}$, where $A_{i}\in \tau$ .  We define
 \[ H_{\mu}(\alpha)=-\sum_{A \in \alpha} \mu(A) \log \mu(A)\]
 to be the \textit{entropy} of $\alpha$  (with the convention 0$\log 0$ = 0).
 
 We denote by $\alpha \vee \beta$ the joint partition $\{A\cap B\hspace{0.1cm}|\hspace{0.1cm} A\in \alpha, B\in \beta \}.$\\
 Let $T^{-1}(\alpha)=\{T^{-1}(A) \hspace{0.1cm}|\hspace{0.1cm}A\in \alpha\}$. We define
 \[ h(\mu, \alpha)=\lim_{n\rightarrow \infty}\frac{1}{n} H_{\mu}(\bigvee_{j=0}^{n-1}T^{-1}(\alpha))\]
 to be the entropy of $T$ relative to $\alpha$ \footnote{Limits exist by subadditivity.}.
 
 Then the metric entropy of $\mu$ is defined as
 \[h_{\mu}(T)=\sup h(\mu, \alpha),\]
 where the supremum is taken over all countable partitions $\alpha$ with $H_{\mu}(\alpha)<\infty.$
 
  We can define the topological pressure by the following \textit{variational principle}. It was proved by Cao, Feng and Huang \cite{CFH}.
 \begin{thm}[{{\cite[Theorem 1.1]{CFH}}}]\label{var1} 
Let $(X,T)$ be a TDS such that $h_{top}(T)<\infty$. Suppose that $\Phi=\{\log\phi_{n}\}_{n=1}^{\infty}$ is a subadditive potential on the compact metric space $X$. Then
\[ P(T, \Phi)=\sup\{h_{\mu}(T)+\chi(\mu, \Phi) \]
\[: \mu \in \mathcal{M}(X,T) , \chi(\mu, \Phi) \neq -\infty \}.\]
\end{thm}
For $t\in \R_{+}$, let us denote $P(T, t\Phi)=P(t).$ 
\begin{thm}[{{\cite[Theorem 1.2]{FH}}}] \label{max1}
Let $(X,T)$ be a TDS such that $h_{top}(T)<\infty$. Assume that $\Phi=\{\log \phi_{n}\}_{n=1}^{\infty}$ is a  subadditive potential on the compact metric space $X$ which satisfies $\beta(\Phi)>-\infty$ .

   Then the pressure function $P(t)$ is  a  continuous real convex function on $(0,\infty)$. Furthermore, $P^{'}(\infty) := \lim_{t \rightarrow \infty} \frac{P(t)}{t}=\beta(\Phi)$.
\end{thm}
Let $t\in \R_{+}$, we denote by $\Eq(t)$ the collection of invariant measures $\mu$ such that 
\[ h_{\mu}(T)+t.\chi(\mu, \Phi)= P(t).\]

If $\Eq(t)\neq \emptyset$, then each element $\Eq(t)$ is called an $\textit{equilibrium state}$ for $t\Phi$. 

\begin{prop}[{{\cite[Theorem 3.3]{FH}}}]\label{eq11}
Let $(X, T)$ be a TDS such that the  entropy  map $\mu\mapsto h_{\mu}(T)
$ is upper semi-continuous and $h_{top}(T)<\infty$. Suppose that  $\Phi=\{\log \phi_{n}\}_{n=1}^{\infty}$ is a subadditive potential on the compact metric space $X$ which satisfies $\beta(\Phi)>-\infty$. For any $t>0$, $\Eq(t)$ is a non-empty  compact  convex  subset  of $\mathcal{M}(X,T)$, and every extreme point of $\Eq(t)$ is an ergodic  measure. Moreover,
\[\partial P(t) =\{\chi(\mu_{t}, \Phi) : \mu_{t}\in \Eq(t)\}.\]
\end{prop}
\begin{thm}[{{\cite[Proposition 3.2]{FH}}}]\label{interval}
Suppose that $\Phi=\{\log \phi_{n}\}_{n=1}^{\infty}$ is a subadditive potential on a TDS $(X, T)$. Assume that $h_{top}(T)<\infty$ and $\beta(\Phi)>-\infty.$ Then
\[\partial P(\R_{+})\subseteq (-\infty, \beta(\Phi)],\]
where $\partial P(\R_{+})$ defined in (\ref{dense}).
\end{thm}
We denote by $\mathcal{M}(X)$ the space of all Borel probability measure on X with weak$^{\ast}$ topology.
\begin{thm}[{{\cite[Lemma 2.3]{CFH}}}] \label{con1}
Suppose $\{\nu_{n}\}_{n=1}^{\infty}$ is a sequence in $\mathcal{M}(X)$ and $\Phi=\{\log \phi_{n}\}_{n=1}^{\infty}$ is a subadditive potential on a TDS $(X, T)$.  We form the new sequence $\{\mu_{n}\}_{n=1}^{\infty}$ by $\mu_{n}=\frac{1}{n}\sum_{i=0}^{n-1}\nu_{n}oT^{i}$.  Assume that $\mu_{n_{i}}$ converges to $\mu$ in $\mathcal{M}(X)$ for some subsequence $\{n_i\}$ of natural numbers.  Then $\mu\in \mathcal{M}(X,T)$ and 
\begin{equation}\label{add1}
\limsup_{i \rightarrow \infty} \frac{1}{n_{i}}\int \log \phi_{n_{i}}(x)d\nu_{i}(x)\leq \chi (\mu, \Phi).
\end{equation} 
\end{thm}
\section{Proof of the Theorem \ref{main1}}
We start the proof of Theorem \ref{main1} $(i)$ with the key Proposition \ref{eq11} which tells us that the subderivative of the topological pressure for a subadditive potential is equal to the Lyapunov exponent of the equilibrium state. Let $(X,T)$ be a TDS such that the  entropy  map $\mu\mapsto h_{\mu}(T)
$ is upper semi-continuous and $h_{top}(T)<\infty$. Suppose that $\Phi=\{\log \Phi_{n}\}_{n=1}^{\infty}$ is a subadditive potential on the compact metric space $X$ which satisfies $\beta(\Phi)>-\infty$. We write Theorem \ref{main1} $(i)$ as follows. 
\begin {thm}\label{zero1}
 For each $t>0$, the family of equilibrium measures $(\mu_{t})$, has a weak$^{\ast}$ accumulation point $\mu$ as $t \rightarrow \infty.$ Any such accumulation point is a Lyapunov maximizing measure for $\Phi$. Moreover,
\[ \chi(\mu, \Phi)=\lim_{t\rightarrow \infty} \chi(\mu_{t}, \Phi).\]

\end{thm}
\begin{proof}
It is obvious that $(\mu_{t})$ has at least one accumulation point, let us call it $\mu$. By Theorem $\ref{max1}$, $P(t)$ is convex , then we have $\partial P(t)=\{\chi(\mu_{t}, \Phi)\}$ by Proposition \ref{eq11}. Moreover, since $P(t)$ is convex for $t>0$, $t \mapsto \chi(\mu_{t}, \Phi)$ is non-decreasing and bounded above\footnote{This follows from subadditivity.}.
\\
It follows that 
\[ \lim_{t \rightarrow \infty} \partial P(t)=\lim_{t \rightarrow \infty} \chi(\mu_{t}, \Phi) \hspace{0.2cm} \textrm{exists and is finite}.\]
Since Lyapunov exponents are upper semi-continuous
\[\lim_{t \rightarrow \infty} \chi(\mu_{t}, \Phi) \leq \chi(\mu, \Phi) .\]
By the definition of $\Eq(t)$, 
\begin{equation}\label{1side}
 \chi(\mu_{t}, \Phi) +\frac{h_{\mu_{t}}(T)}{t} \geq \chi(\mu, \Phi) +\frac{h_{\mu}(T)}{t}.
\end{equation} 

Since the TDS $(X,T)$ has finite topological entropy, so when $t\rightarrow \infty$,  (\ref{1side}) implies
\[\lim_{t \rightarrow \infty} \chi(\mu_{t}, \Phi) \geq \chi(\mu, \Phi) .\]

Now, we shall show that $\mu$ is a Lyapunov maximizing measure.

 By contradiction, let us assume that there exists $\nu$ with $\chi(\nu, \Phi)-\chi(\mu , \Phi)= \kappa>0.$ One can  define the affine map $T_{\nu}:\R_{+}\rightarrow \R$ by $T_{\nu}(t)=h_{\nu}(T) +t \chi(\nu, \Phi)$. We know that $t \mapsto \chi(\mu_{t}, \Phi)$ is a function which increases to its limit $\chi(\mu, \Phi)$, so 
\begin{align*}
&\chi(\mu, \Phi)\geq  \chi(\mu_{t_{\ast}}, \Phi)=\partial^{e} P(t_{\ast}) ,\textrm{where}\hspace{0.1cm}t_{\ast}=t_{-}\hspace{0.1cm} \textrm{or} \hspace{0.12cm}t_{+},\\
& \textrm{and}\hspace{0.2cm}T_{\nu}^{'}(t)=\chi(\nu, \Phi)=\chi(\mu, \Phi)+\kappa\geq \partial^{e} P(t_{\ast})+\kappa.
\end{align*}
Consequently,
 $h_{\nu}(T)+t\chi(\nu, \Phi) > P(t)$ for all sufficiently large $t>0$ that contradicts our assumption. So, $\mu$ is a Lyapunov maximizing measure.
 
 Moreover, our proof implies that $\beta(\Phi)$ can be approximated by 
 Lyapunov exponents of equilibrium measures of a subadditive potential $t\Phi$.
\end{proof}
 Theorem \ref{main1} $(ii)$ is obtained by combining Lemmas \ref{zero2} and \ref{compact1} below.
\begin{lem}\label{zero2}
The maps $t\mapsto h_{\mu_{t}}(T)$ and $t\mapsto P(t\Phi-t \beta(\Phi))$ are non-increasing and bounded below on the interval $(0, \infty)$. Moreover, we have
\[ \lim_{t \rightarrow \infty} h_{\mu_{t}}(T)=\lim_{t\rightarrow \infty} P(t\Phi-t\beta(\Phi))\geq \sup_{\nu \in \mathcal{M}_{\max}(\Phi)} h_{\nu}(T) .\]
\end{lem}
\begin{proof}
The map $t \mapsto P(t\Phi- t\beta (\Phi))$ is convex. By the definition of $\beta(\Phi)$,
\[ \chi(\mu_{t}, \Phi) \leq \beta(\Phi)\hspace{0.2cm}\textrm{for all}\hspace{0.1cm}\mu_{t}\in \Eq(t).\]

We assume that $P(t)=P(t\Phi)$. By the definition of the topological pressure, $P(t \Phi- t\beta(\Phi))=P(t\Phi)- t\beta(\Phi)$. Then, \[\partial^{e} P(t_{\ast}\Phi-t_{\ast}\beta(\Phi))=\partial^{e} P(t_{\ast}\Phi)-\beta(\Phi)=\chi(\mu_{t_{\ast}},\Phi) -\beta(\Phi) \leq 0,\] where $t_{\ast}=t_{-}\hspace{0.1cm} \textrm{or} \hspace{0.12cm}t_{+}$. Thus, $P(t\Phi -t\beta(\Phi))$ is non-increasing by Theorem \ref{incre1}. We are going to show that $t\mapsto h_{\mu_{t}}(T)$ is non-increasing. Since $\mu_{t}$ is an equilibrium measure,
\[ h_{\mu_{t_{\ast}}}(T)=P(t)-t\partial^{e} P(t_{\ast}).\] 

 For $0<x<y$ we have \[\partial^{e} P(x_{\ast}) \leq \frac{P(y)-P(x)}{y-x} \leq  \partial^{e} P(y_{\ast}),\] so $\displaystyle {P(y)-P(x)} \leq y \partial^{e} P(y_{\ast})-x \partial^{e} P(y_{\ast})\leq y \partial^{e} P(y_{\ast})-x \partial^{e} P(x_{\ast})$, and then \[P(y)-y \partial^{e} P(y_{\ast})\leq P(x)-x \partial^{e} P(x_{\ast}).\] 

Since $t \mapsto h_{\mu_{t}}(T)$ and $ t \mapsto P(t\Phi-t\beta(\Phi)) \geq 0$ are non-increasing and non-negative, we conclude that $\lim_{t \rightarrow \infty} h_{\mu_{t}}(T)$ and $\lim_{t\rightarrow \infty}P(t\Phi-t\beta(\Phi))$ both exist. This implies that the limit
\[\lim_{t\rightarrow \infty} t\partial^{e} P(t)-t\beta (\Phi)=\lim_{t\rightarrow \infty}(P(t\Phi-t\beta(\Phi))-h_{\mu_{t}}(T))\]
exists. Then,
\[\lim_{t\rightarrow \infty}h_{\mu_{t}}(T)=\lim_{t\rightarrow \infty}P(t\Phi-t\beta(\Phi)).\]
The last part follows from the variational principle.
\end{proof}

\begin{lem}\label{compact1}
 $\mathcal{M}_{\max}(\Phi)$ is  compact,  convex and  nonempty,  and  its  extreme  points  are  precisely  its  ergodic  elements.

\end{lem}
\begin{proof}
See \cite[Appendix A]{M3}.
\end{proof}
We write Theorem \ref{main1} $(ii)$ as follows.
\begin{thm}\label{cont1en}
$h_{\mu }(T)=\lim_{t \rightarrow \infty} h_{\mu_{t}}(T)=\max\{h_{\nu}(T), \nu \in \mathcal{M}_{\max}(\Phi)\}$.
\end{thm}
\begin{proof}
By Theorem \ref{zero1} and Lemmas \ref{zero2} and \ref{compact1},
\[h_{\mu}(T)\leq \max_{\nu \in \mathcal{M}_{max}(\Phi)} h_{\nu}(T) \leq \lim_{t\rightarrow \infty} h_{\mu_{t}}(T),\]
the reverse inequality follows from upper semi-continuity of entropy.
\end{proof}
\begin{rem} Let $(T, \mathcal{A})$ be a locally constant cocycle. Then, one can prove Theorem \ref{main1} for Gibbs measures under the assumption that $(T, \mathcal{A})$ is irreducible (see \cite{FK}). Moreover, if $T:X\rightarrow X$ is a  mixing subshift of finite type  and $\mathcal{A}:X \rightarrow GL(d, \R)$ is a H\"older continuous function, then one can prove Theorem \ref{main1} for Gibbs measures under the generic assumption on $(T, \mathcal{A})$ (see \cite{P}).
\end{rem}
\begin{rem}\label{rem1} Let $ \vec{q} = (q_{1}, ..., q_{d})\in \R_{+}^{d}$, and $\vec{\Phi}=(\Phi_{1},...,\Phi_{d})=(\{\log \phi_{n,1}\}_{n=1}^{\infty},...,\\\{\log \phi_{n,d}\}_{n=1}^{\infty})$. Assume that $\vec{q}.\vec{\Phi}=\sum_{i=1}^{d} q_{i}\Phi_{i}$ is a subadditive potential $\{q_{i}\log \phi_{n, i}\}_{n=1}^{\infty}.$ We can write topological pressure, and maximal Lyapunov exponent of $\vec{\Phi}$, respectively
\[ P(\vec{q})=P(T,\vec{q}.\vec{\Phi}),\hspace{0.3cm} \beta(\vec{\Phi})=\beta(\sum_{i=1}^{d} \Phi_{i}).\]
  Feng and Huang \cite{FH} proved the higher dimensional versions of Theorem \ref{max1}, Proposition \ref{eq11}, and Theorem \ref{interval}.  So, one can obtain the higher dimensional versions of Theorem \ref{main1} by using \cite{FH}. 
\end{rem}

\section{Proof of the Theorem \ref{main3}}

In this section, we consider locally constant cocycles and we prove Theorem \ref{main3}. We use the continuity of Lyapunov exponents (Theorem \ref{cont1}), the Anosov closing property and Theorem \ref{con1} for the proof.

Let $(T, \mathcal{A})$ be the locally constant cocycle which is defined in Example \ref{ex1} and $\mathcal{A}:X\rightarrow GL(2,\R)$. We denote $\chi(\mu, \mathcal{A}):=\chi(\mu, \Phi)$, where $\phi_{n}=\|\mathcal{A}_{n}\|$.

Bocker and Viana \cite{BV} proved the continuity of Lyapunov exponents for two 
dimensional locally constant cocycles. In order to state the result of Bocker and Viana, we denote by $\triangle_{k}$ the collection of strictly positive probability vectors in $\R^k$ for $k\geq 2$. We denote by $X$ the full shift space over $k$ symbols. For $p= (p_{1},...,p_{k})\in \triangle_{k}$, let $\mu$ be the associated Bernoulli product measure on $X$.
\begin{thm}[{{\cite[Theorem B]{BV}}}]\label{cont1}
For every $\epsilon >0$ there exist $\delta >0$ and a weak$^{\ast}$ neighborhood $V$ of  $\mu$ in the space of probability measures on $GL(2,\R)$ such that for every probability measure $\mu^{'}\in V$ whose support is contained in the $\delta$-neighborhood of the support of $\mu$, we have
\[|\chi(\mu, \mathcal{A})-\chi(\mu^{'}, \mathcal{A}^{'})|< \epsilon.\] 
\end{thm}
Now, we can prove Theorem \ref{main3}.
\begin{thm}\label{speed}
 Suppose that $T$ satisfies the Anosov closing property. Then the maximal Lyapunov exponent $\beta(\Phi)$ can be approximated by Lyapunov exponents of measures supported on periodic orbits.
\end{thm}
\begin{proof}
 Let $\mu$ be an ergodic maximizing measure, that is $\beta(\Phi)=\chi(\mu, \Phi)$. 

Let $x$ be a generic point for $\mu$. Then there exists $\mu_{n,x}:=\frac{1}{n}\sum_{j=0}^{n-1}\delta_{T^{j}(x)}$, where $\delta_{x}$ is the Dirac measure at the point $x$, so that $\mu_{x,n}  \rightarrow \mu $. According to Theorem \ref{con1}, and (\ref{invariant1}) \[\lim_{n\rightarrow \infty}\frac{1}{n}\log\phi_{n}(x) =\chi(\mu, \Phi).\]

Let $p \in X$ be a periodic point associated to $\epsilon,C, \delta$ and
\[\{x, T(x),...,T^{n-1}(x)\}\] by the Anosov closing property. Denote by $\mu_{p}:=\frac{1}{n}\sum_{j=0}^{n-1}\delta_{T^{j}(p)}$  the ergodic $T-$invariant measure supported on the orbit of $p$.
 \begin{lem}\label{lem1}
  $\mu_{p}\rightarrow \mu$ in weak$^{\ast}$topology. 
  \end{lem}
  \begin{proof}
  We will use the Anosov closing property.
  
   Assume that $(f_{m})$ is a sequence of continuous functions. The periodic orbit $p$ has length $n$ and is close to the initial segment of the orbit of $x$. Since the $f_{m}$'s are continuous, the average of $f_{m}$ along the periodic orbit is very close to the average of $f_{m}$ along the first $n$ iterates of $x$, and that is very close to $\int f_{m} d\mu$ by the genericity condition. Then, for $n$ large enough, we get longer and longer periodic orbits, approaching $x$ more closely, and we obtain the convergence of the measures to $\mu$. 
  \end{proof}
We now use Lemma \ref{lem1} to finish the proof. By the Anosov closing property, the periodic point $p$ is close to $x$, with iterates also close to the iterates of $x$. Therefore, 
 Theorem \ref{cont1} implies for every $\epsilon>0$
\begin{equation}\label{eq111}
\chi(p, \Phi)=\lim_{n\rightarrow \infty}\frac{1}{n}\log\phi_{n}(p)= \lim_{n\rightarrow \infty}\frac{1}{n}\log\phi_{n}(x) +\epsilon.
\end{equation}
Applying Lemma \ref{lem1}, Theorem \ref{con1}  and (\ref{eq111}), we obtain 
 \[\chi(p, \Phi)=\lim_{n\rightarrow \infty}\frac{1}{n}\log\phi_{n}(p)=\chi(\mu,\Phi)+ \epsilon=\beta(\Phi)+ \epsilon. \]
\end{proof}
\begin{rem}
Avila, Eskin and Viana \cite{AEV} announced recently that the Theorem \ref{cont1} remains true in arbitrary dimensions. By their result, the proof given for Theorem \ref{cont1} works for arbitrary dimensions.
\end{rem}
\begin{rem} Morris \cite{M2} showed that the speed of convergence of Theorem \ref{speed} is always superpolynomial for locally constant cocycles. Moreover, Bochi and Garibaldi \cite{BG} showed that it is true for general cocycles under certain assumptions.
\end{rem}

\bibliography{ref}

\end{document}